\documentclass[a4paper]{article}

\usepackage{
amsmath,
amsthm,
amscd,
amssymb,
}
\usepackage{comment}
\usepackage{tikz}
\usetikzlibrary{positioning,arrows,calc}
\usepackage{xspace}

\usepackage{graphicx}

\usepackage{url}

\theoremstyle{plain}
\newtheorem{lemma}{Lemma}
\newtheorem{definition}{Definition}

\newtheorem{theorem}{Theorem}

\newtheoremstyle{derp}
{3pt}
{3pt}
{}
{}
{\upshape}
{:}
{.5em}
{}
\theoremstyle{derp}

\newcommand{\R}{\mathbb{R}}

\newcommand{\Z}{\mathbb{Z}}

\newcommand{\varx}{\mathbf{x}}
\newcommand{\vary}{\mathbf{y}}
\newcommand{\bol}[1]{\mathbf{#1}}

\title{A note on directional closing}

\author{
Ville Salo \\
vosalo@utu.fi
}

\begin{document}
\maketitle

\begin{abstract}
We show that directional closing in the sense of Guillon-Kari-Zinoviadis and Franks-Kra is not closed under conjugacy. This implies that being polygonal in the sense of Franks-Kra is not closed under conjugacy.
\end{abstract}

\section{Introduction}

Say a set $S \subset \R^2$ \emph{codes} a set $T \subset \R^2$ for a subshift $X \subset A^{\Z^2}$ if
\[ x, y \in X \wedge x|_{S \cap \Z^2} = y|_{S \cap \Z^2} \implies x|_{T \cap \Z^2} = y|_{T \cap \Z^2}. \]
A set $S$ is \emph{expansive} if it codes $\R^2$.

The following definitions of directional closing are due to Guillon, Kari and Zinoviadis \cite{GuKaZi15} and John Franks and Bryna Kra \cite{FrKr19} (with slightly different terminology).

\begin{definition}
Let $H = \{(a,b) \in \R^2 \;|\; b < 0 \vee (b = 0 \wedge a < 0) \}$. A subshift $X \subset A^{\Z^2}$ is \emph{right-closing in direction $\bol v$}, where $\bol v \in S^1$ is a direction represented by a point on the unit sphere, if $R(H)$ codes $\overline{R(H)}$, where $R$ is the linear rotation that takes $(0,1)$ to $\bol v$. Left-closing is defined symmetrically, and a subshift is \emph{closing} in direction $\bol v$ if it is either left- or right-closing in that direction, and \emph{bi-closing} if it is both.
\end{definition}

Note that with this definition a subshift is closing in an irrational direction $\bol v$ if and only if it is \emph{deterministic} in that direction, meaning the half-plane $R(H)$ is expansive. For rational directions, bi-closing does not imply determinism. For the spacetime subshift of a surjective CA, being directionally left-, right- or bi-closing in the opposite direction of where time flows agrees with the standard definition for cellular automata \cite{Ku09}.

We mention the following lemma (not original to this note) which is one motivation for this definition. We say $T \subset \Z^2$ is an \emph{extremally permutive} shape\footnote{For any two corners $\bol u, \bol v$, the contents of $T \setminus \{\bol u, \bol v\}$ determine a partial permutation between possible values at $\bol u$ and $\bol v$.} for a subshift $X \subset A^{\Z^2}$ if $T \setminus \{\bol v\}$ codes $T$ in $X$ whenever $\bol v$ is an extremal point of the convex hull of $T$ in $\R^2$.

\begin{lemma}
\label{lem:AllBiclosing}
If there exists an extremally permutive shape for a subshift $X$, then all directions of $X$ are bi-closing.
\end{lemma}

The converse is also true \cite{GuKaZi15,FrKr19}.

Such shapes appear frequently in multidimensional symbolic dynamics and the theory of cellular automata, for example in algebraic symbolic dynamics \cite{Sc95}, in the context of Nivat's conjecture \cite{CyKr15a,KaSz15}, and in the spacetime subshifts of bipermutive CA. The tiling problem stays undecidable for SFTs admitting $\{(0,0), (1,0), (0,1), (1,1)\}$ as a corner-deterministic shape \cite{Lu09a}.

In this note, we show that directional closing is (not surprisingly) not conjugacy-invariant. \emph{Left} means $(-1,0)$.

\begin{theorem}
\label{thm:NondeterministicsNonclosing}
The conjugacy class of the Ledrappier subshift contains a subshift that is neither left- nor right-closing to the left.
\end{theorem}

We only conjugate one direction to be non-closing, but readers familiar with marker techniques will have no trouble extending this to show that all three non-deterministic directions can be simultaneously conjugated to be non-closing (in both directions).

In a recent paper \cite{FrKr19}, Franks and Kra study subshifts admitting an extremally permutive shape. In their terminology, such subshifts are called \emph{polygonal}. They ask in \cite[Question~1]{FrKr19} whether polygonal subshifts are closed under conjugacy. Our example above solves the question in the negative.


\begin{theorem}
\label{thm:NotPolygonal}
The conjugacy class of the Ledrappier subshift contains a non-polygonal subshift.
\end{theorem}

\section{Proofs}

\begin{proof}[Proof of Theorem~\ref{thm:NondeterministicsNonclosing}]
Write $\Z_2$ for the ring $\Z/2\Z$. There is a standard way to see $\Z_2^{\Z^2}$ as a $\Z_2[\varx, \varx^{-1}, \vary, \vary^{-1}]$-module, and the \emph{Ledrappier subshift} $X \subset \Z_2^{\Z^2}$ is the variety $\{x \;|\; px = 0\}$ where $p = 1 + \varx + \vary$.

We fix the geometric convention that $X$ is the subshift where looking through each window of shape $\begin{tikzpicture}[baseline = 0, scale = 0.3] \draw (0,0) rectangle (1,1); \draw (1,0) rectangle (2,1); \draw (0,1) rectangle (1,2); \end{tikzpicture}$ you see an even number of $1$s.

We add another layer of information that implements the possibility of slight skewing of the location of bits, and show that the left direction can be made non-closing both left and right. 

We define another subshift $Y \subset (\Z_2^2)^{\Z^2}$ where we allow either one or two bits per cell. We will define this as the image of $X$ under an injective morphism $f : \Z_2^{\Z^2} \to (\Z_2^2)^{\Z^2}$, so in particular it conjugates $X$ onto $Y$.

First define a one-dimensional injective morphism $g : \Z_2^\Z \to (\Z_2^2)^\Z$ by $g(x)_i = g_{\mathrm{loc}}(x_i, x_{i+1})$ where 
\[ g_{\mathrm{loc}}(a,b) = \left\{ \begin{array}{ll}
(a, 0) & \mbox{if } b = 0 \\
(0, 1) & \mbox{if } b = 1 \\
\end{array} \right. \]
This is just the $2$-blocking representation \cite{LiMa95} followed by the symbolwise projection that maps $(1,1)$ to $(0,1)$ and fixes other symbols. A left inverse is obtained by projecting to the second coordinate and shifting, so $g$ is indeed injective.

We define $f$ by applying $g$ on every row, or in formulas $f(x)_{\vec v} = g(x_{\vec v_2})_{\vec v_1}$ where for $x \in \Z_2^{\Z^2}$ we write $x_i \in \Z_2^\Z$ for row extraction, i.e. $(x_i)_j = x_{(j,i)}$.

Now consider the (obvious infinite extensions of the) following configurations

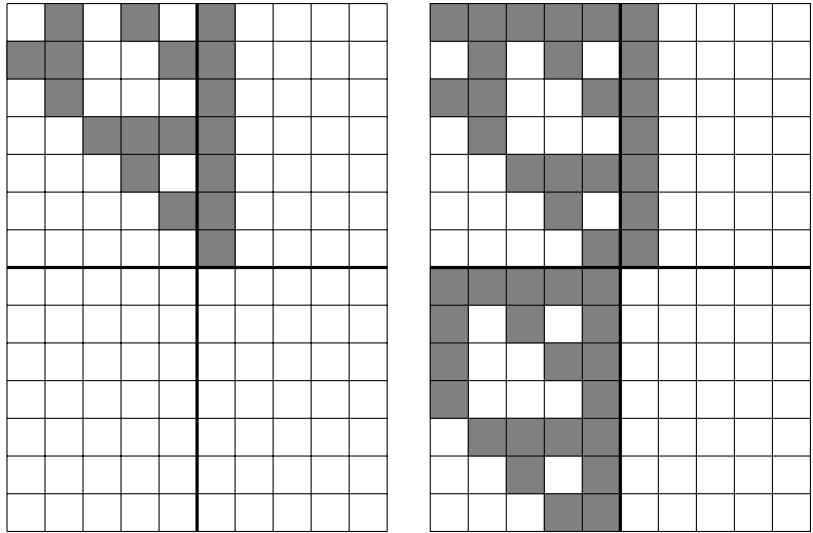
\begin{center}
\begin{tikzpicture}[baseline = 0, scale = 0.5]
\foreach \x/\y in { 1/13, 3/13, 5/13, 0/12, 1/12, 4/12, 5/12, 1/11, 5/11, 2/10, 3/10, 4/10, 5/10, 3/9, 5/9, 4/8, 5/8, 5/7
}{\draw[fill, gray] (\x,\y) rectangle (\x+1,\y+1);}
\draw (0,0) grid (10,14);
\draw[very thick] (5,0) -- (5,14);
\draw[very thick] (0,7) -- (10,7);
\end{tikzpicture} \;\;\;
\begin{tikzpicture}[baseline = 0, scale = 0.5]
\foreach \x/\y in { 0/13, 1/13, 2/13, 3/13, 4/13, 5/13, 1/12, 3/12, 5/12, 0/11, 1/11, 4/11, 5/11, 1/10, 5/10, 2/9, 3/9, 4/9, 5/9, 3/8, 5/8, 4/7, 5/7, 0/6, 1/6, 2/6, 3/6, 4/6, 0/5, 2/5, 4/5, 0/4, 3/4, 4/4, 0/3, 4/3, 1/2, 2/2, 3/2, 4/2, 2/1, 4/1, 3/0, 4/0
}{\draw[fill, gray] (\x,\y) rectangle (\x+1,\y+1);}
\draw (0,0) grid (10,14);
\draw[very thick] (5,0) -- (5,14);
\draw[very thick] (0,7) -- (10,7);
\end{tikzpicture}
\end{center}

\noindent where the color gray denotes $1$. The conjugate images in $Y$ are (respectively)

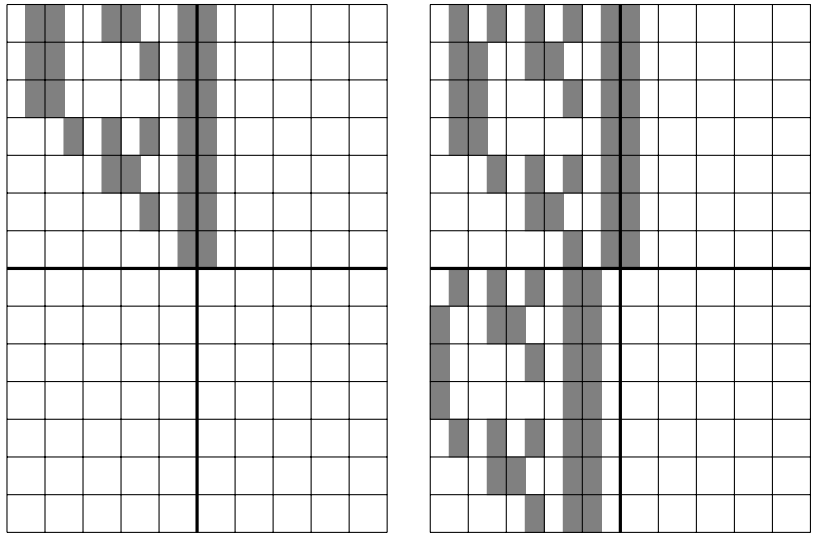
\begin{center}
\begin{tikzpicture}[baseline = 0, scale = 0.5]
\foreach \x/\y in { 1/13, 3/13, 5/13, 1/12, 5/12, 1/11, 5/11, 5/10, 3/9, 5/9, 5/8, 5/7}{\draw[fill, gray] (\x,\y) rectangle (\x+0.5,\y+1);}
\foreach \x/\y in { 0/13, 2/13, 4/13, 0/12, 3/12, 4/12, 0/11, 4/11, 1/10, 2/10, 3/10, 4/10, 2/9, 4/9, 3/8, 4/8, 4/7}{\draw[fill, gray] (\x+0.5,\y) rectangle (\x+1,\y+1);}
\draw (0,0) grid (10,14);
\draw[very thick] (5,0) -- (5,14);
\draw[very thick] (0,7) -- (10,7);
\end{tikzpicture} \;\;\;
\begin{tikzpicture}[baseline = 0, scale = 0.5]
\foreach \x/\y in { 5/13, 1/12, 3/12, 5/12, 1/11, 5/11, 1/10, 5/10, 5/9, 3/8, 5/8, 5/7, 4/6, 0/5, 2/5, 4/5, 0/4, 4/4, 0/3, 4/3, 4/2, 2/1, 4/1, 4/0}{\draw[fill, gray] (\x,\y) rectangle (\x+0.5,\y+1);}
\foreach \x/\y in { 0/13, 1/13, 2/13, 3/13, 4/13, 0/12, 2/12, 4/12, 0/11, 3/11, 4/11, 0/10, 4/10, 1/9, 2/9, 3/9, 4/9, 2/8, 4/8, 3/7, 4/7, 0/6, 1/6, 2/6, 3/6, 1/5, 3/5, 2/4, 3/4, 3/3, 0/2, 1/2, 2/2, 3/2, 1/1, 3/1, 2/0, 3/0}{\draw[fill, gray] (\x+0.5,\y) rectangle (\x+1,\y+1);}
\draw (0,0) grid (10,14);
\draw[very thick] (5,0) -- (5,14);
\draw[very thick] (0,7) -- (10,7);
\end{tikzpicture}
\end{center}
where we color the left and right half of each cell gray or white depending on the left- and rightmost bit, respectively, again gray means 1. The picture plainly shows that $Y$ is not left-closing to the left. Similarly one can show it is not right-closing to the left. 
\end{proof}

\begin{proof}[Proof of Theorem~\ref{thm:NotPolygonal}]
By Lemma~\ref{lem:AllBiclosing}, a polygonal subshift is bi-closing in all directions. Thus, the subshift $Y$ constructed in the previous proof is not polygonal.
\end{proof}

\bibliographystyle{plain}
\bibliography{../../../bib/bib}{}

\end{document}